\documentclass[12pt,a4paper]{article}
\usepackage{microtype,amssymb,amsmath,amsthm}
\usepackage[margin=1.4in]{geometry}
\usepackage[english]{babel}
\usepackage[utf8x]{inputenc}
\usepackage[colorinlistoftodos]{todonotes}
\usepackage[colorlinks=true, allcolors=blue]{hyperref}
\author{Johan Andersson\thanks{Email:johan.andersson@oru.se \, Address:Department of Mathematics, School of Science and Technology, {\"O}rebro University, {\"O}rebro, SE-701 82 Sweden. }}
 
\title{Voronin universality on the abscissa of absolute convergence}

\theoremstyle{plain} 
\newtheorem{thm}{Theorem}  
\newtheorem{lem}{Lemma}

\newcommand{\abs}[1]{{\left| {#1} \right|}} 
\newcommand{\p}[1]{{\left(
     {#1} \right)}}
\newcommand{\C}{{\mathbb C}} 
\newcommand{\R}{{\mathbb R}}
\renewcommand{\Re}{\operatorname{Re}} 
 
   \date{}
\begin{document}

\maketitle        
\begin{abstract} 
   We prove that  the Voronin universality theorem for the Riemann zeta-function extends to the line $\Re(s)=1$   if in addition to vertical shifts we also allow scaling and adding a sufficiently large constant.
\end{abstract}

\maketitle
\section{Introduction and main results}
Voronin \cite{Voronin1} proved that
\begin{gather*}
  \{(\zeta(1+it),\zeta'(1+it),\ldots,\zeta^{(n)}(1+it))| \, t \in \R \}
\end{gather*}
is dense in $\C^n$. In contrast we proved \cite{Andersson1} that the related more general theorem, the Voronin universality theorem \cite{Voronin2} which on the one-line would say that any continuous function $f(t)$ on an interval $[0,H]$ could be approximated in sup-norm to any given accuracy by shifts $\zeta(1+iT+it)$ does not hold. This is a consequence of \cite[Theorem 8]{Andersson1}
\begin{gather} \label{uu}
 \inf_T  \max_{T \leq t \leq T+\delta} \abs{\zeta(1+it)} = \frac{\pi^2 e^{-\gamma}}{24} \delta+O(\delta^3),  
\end{gather}
since this implies that the function $f(t)=0$ can not be approximated to any given accuracy by the Riemann zeta-function. For further discussion on how these results are related see \cite[pp. 2--3]{Andersson2}. We note  that the bound in \eqref{uu} does depend on  $\delta$ and if we also allow scaling it follows from \eqref{uu} that the function $f(t)=0$ may be approximated by the Riemann-zeta function on the line $\Re(s)=1$. Indeed if $f(t)=0$ on the interval $[0,1]$ then given $\varepsilon>0$ there exist some $\delta>0$ and some $T>0$ such that 
\begin{gather*}
   \max_{0 \leq t \leq 1} 
   \abs{\zeta(1+iT+i \delta t)-f(t)}< \varepsilon.
\end{gather*}
The purpose of this paper is to extend this observation to any continuous function $f$ on a compact set $K$ with connected complement, where $f$ is  analytic in the interior of $K$. In doing so we also need to introduce a constant term.
   \begin{thm} \label{th} Let
 $K \subset \C$ be a compact set with connected complement, and suppose that $f$ is any continuous function on $K$ that is analytic  in the interior of $K$.  Then for any $\varepsilon>0$ there exist $C_0,\delta_0>0$ such that for any $0<\delta\leq \delta_0$  and  $|C|>C_0$ then
\begin{gather*}
 \liminf_{T \to \infty} \frac 1 T \mathop{\rm meas} \left \{t \in [0,T]:\max_{s \in K} \abs{\zeta(1+it+\delta s)+C-f(s)}<\varepsilon \right \}>0. 
\end{gather*}
\end{thm}

The Voronin universality theorem \cite{Voronin2} allow us to choose $\delta=1$ and $C=0$ for any $\varepsilon>0$ in Theorem \ref{th} when
$K \subset \{s \in \C: -\frac 1 2<\Re(s)<0\}$ and  $f$ is zero-free on  $K$. 
Which $C$ we can choose in Theorem 1 depends on the analytic properties of the function $f$. We can choose  $C=0$ for all $\varepsilon>0$ if $\log f$ is, up to the addition of a constant, the Laplace transform of a function bounded by $x^{-1}$. 
 \begin{thm} \label{th2} Let $K \subset \C$ be a compact set with connected complement, and suppose that 
 \begin{gather*}
      f(s)=C+\int_0^\infty g(x) e^{-sx}dx, \\ \intertext{for $s \in K$ where}  
        \abs{xg(x)} \leq 1.
 \end{gather*}
  Then for any $\varepsilon>0$ there exists some  $\delta_0>0$ such that for any $0<\delta\leq \delta_0$  then
\begin{gather*}
 \liminf_{T \to \infty} \frac 1 T \mathop{\rm meas} \left \{t \in [0,T]:\max_{s \in K} \abs{\log \zeta(1+it+\delta s)-f(s)}<\varepsilon \right \}>0. 
\end{gather*}
\end{thm}
\section{Proofs of main results}
\subsection{Main Lemmas}
In order to prove our main results  we need some  well-known fact from the theory of universality, which we state in the following convenient form. 
\begin{lem} \label{LALA}
  Let \begin{gather} \label{ep} h(s)=-\sum_{p}  \log\p{1-\frac{a_p}{ p^{s}}} \end{gather} where $|a_p|=1$ and the sum over the primes is convergent to an analytic function $h$ on the half-plane $\Re(s)>\frac 1 2$. Then for any $\varepsilon>0$ and compact set $K \subset \{s \in \C: \Re(s)>\frac 1 2\}$ we have that
  \begin{gather*}
 \liminf_{T \to \infty} \frac 1 T \mathop{\rm meas} \left \{t \in [0,T]:\max_{s \in K} \abs{\log \zeta(s+it)-h(s)}<\varepsilon \right \}>0. 
\end{gather*}
\end{lem}
\begin{proof}
 This follows\footnote{It would also follow from \cite[Theorem 4.3]{Steuding} unless the result was artifically restricted to a half-strip $\sigma<\Re(s)<1$ (its proof holds in the more general context).} from e.g \cite[Theorem 4.12]{Steuding}. 
\end{proof} 
While Lemma \ref{LALA} is similar to the classical Voronin universality theorem, 
one difference is that we do not need to assume that the compact set $K$ has a connected complement\footnote{The condition comes from the application of Mergelyan's theorem in the proof.}. The key differences however are that we may allow $K$ to lie in the full half-plane $\Re(s)>\frac 1 2$, and that it is less clear what functions can be represented by \eqref{ep}. This is where the Pechersky rearrangement theorem is used in the classical argument\footnote{which only holds in the strip $\frac 1 2<\Re(s)<1$}. The main feature of our approach is that we replace the Pechersky rearrangement theorem with the following Lemma, which we will prove in subsection \ref{mainlemma}.
 \begin{lem}
 \label{LA0}
 For any $\varepsilon>0$, compact set $K$  and function  $f$,  satisfying the conditions of Theorem \ref{th2}  there exists some $\delta_0>0$ such that for any $0<\delta \leq \delta_0$ there exist unimodular complex numbers $|a_p|=1$  such that 
  \begin{gather*}
    h(s)=- \sum_{p} \log \p{1-\frac{a_p} {p^{s}}}, \\ \intertext{is convergent to an analytic function $h$ for $\Re(s)>\frac 1 2$ and such that}
     \max_{s \in K} \abs{h(1+\delta s)-f(s) }<\varepsilon.
  \end{gather*}
 \end{lem}
Theorem \ref{th2} follows from Lemma \ref{LALA} and Lemma \ref{LA0}. Theorem \ref{th} follows from Theorem \ref{th2} and the
following Lemma about Laplace-transforms which we will prove in subsection \ref{lemmathree}.
\begin{lem} \label{LA4} Assume that $f$ is a continuous function on a compact set $K$ with connected complement such that $f$ is analytic the interior of $K$. Then given $\varepsilon>0$ there exist some $A,B>0$  and continuous function $g:[A,B] \to \C$ such that if
\begin{gather*}
  G(s) = \int_A^B g(x) e^{-sx} dx, \\ \intertext{then}
  \max_{s \in K} \abs{G(s)-f(s)}<\varepsilon.
\end{gather*}
\end{lem}  

\subsection{Proof of Theorem \ref{th2}.}
 By Lemma \ref{LA0} we may find some  $\delta_0>0$ such that for any $0<\delta \leq \delta_0$ there exists some series
\begin{gather*} h(s)=-\sum_p \log\p{1-\frac{a_p}{p^{s}}},
\end{gather*} such that $|a_p|=1$ that is convergent on the half-plane $\Re(s)>\frac 1 2$ to an analytic function $h$ 
such that
\begin{gather} \label{uu1}
   \max_{s \in K} \abs{h(1+\delta s)-f(s)} <\frac {\varepsilon} 2.
\end{gather}
By using Lemma \ref{LALA} with the compact set $1+\delta K$ which for a sufficiently small $\delta$ lies in the half plane $\Re(s)>\frac 1 2$, it follows that
  \begin{gather} \label{uu2}
 \liminf_{T \to \infty} \frac 1 T \mathop{\rm meas} \left \{t \in [0,T]:\max_{z \in 1+\delta K} \abs{\log \zeta(z+it)-h(z)}<\frac \varepsilon  2 \right \}>0. 
\end{gather}
Our result follows by the change of variable $z=1+\delta s$, the inequalities \eqref{uu1}, \eqref{uu2} and the triangle inequality.
 \qed

\subsection{Proof of Theorem \ref{th}.}
Without loss of generality we assume that $0<\varepsilon<1$. It is clear that
\begin{gather*}
   \log(f(s)-C)=\log \p{-C \p{1- \frac {f(s)}  C}}= \log (-C)+\log \p{1-\frac {f(s)}   {C}},
\end{gather*}
and if we  assume that  \begin{gather} \label{C1def} |C| \geq  1+4 \varepsilon^{-1}  \max_{s \in K} |f(s)|
\end{gather}
 then it follows from  the elementary inequality \begin{gather*} \abs{\log(1+z)-z} \leq 2|z|^2/3, \qquad (|z|<1/4), \end{gather*} that 
\begin{gather} \label{ab1}
\max_{s   \in K} \abs{\frac{f(s)} {C} +\log \p{1-\frac {f(s)} {C}}}<\frac{\varepsilon} {6|C|}.
\end{gather}
By Lemma \ref{LA4} there exists some $0<A<B$ and continuous function  $g:[A,B] \to \C$ such that
\begin{gather} \label{ab0}
 \max_{s \in K}  \abs{G(s)-    f(s)}<\frac{\varepsilon} {6},  \\ \intertext{where}
    G(s)=\int_A^B g(x) e^{-sx} dx. \notag
\end{gather}
Let us also assume that \begin{gather*}  |C| \geq  \max_{A \leq x \leq B} |g(x)|. \end{gather*} Then the
function \begin{gather*} 
   h(s)=\frac {G(s)} {C}+  \log {(-C)}
   \end{gather*}
    satisfies the condition of Theorem  \ref{th2} so that there exists some $\delta_0 >0$ such that if $0<\delta \leq \delta_0$ then  
\begin{gather} \label{ab2}
  \max_{s \in K} \abs{\log \zeta(1+it+\delta s)-\frac{G(s)} C- \log (-C)}<\frac {\varepsilon} {6|C|} 
\end{gather}
holds with a  positive lower measure $0\leq t \leq T$ as $T \to \infty$. By  the inequalities \eqref{ab1}, \eqref{ab0},  \eqref{ab2} and the triangle inequality  it follows that
\begin{gather} \label{ab55}
  \max_{s \in K} \abs{\log \zeta(1+it+\delta s)-\log(f(s)-C)}<\frac {\varepsilon} {2|C|} 
\end{gather}
for such $t$. From the elementary inequality $|e^z-1| < 3|z|/2$ if $|z| \leq \frac 1 2$ it follows that
\begin{gather} \label{rrra}
  |X-Y| = \abs{Y} \cdot \abs{e^{\log(X/Y)} -1} < 3\abs{Y}/2 \cdot \abs{\log X-\log Y},
\end{gather}
when $\abs{\log X-\log Y} \leq \frac 1 2$. Now let $X=\zeta(1+it+\delta s)$ and $Y= f(s)-C$. From \eqref{C1def} it follows that $|Y| \leq 5|C|/4$ for $s \in K$, and from \eqref{ab55} it follows that $\abs{\log X-\log Y} \leq \frac 1 2$. Thus from \eqref{ab55} and \eqref{rrra} it follows that
\begin{gather*}
  \max_{s \in K} \abs{ \zeta(1+it+\delta s)+C-f(s)} < \frac 3 2   \cdot \frac {5|C|} 4 \cdot \frac{\varepsilon}{2|C|} <\varepsilon  
\end{gather*}
holds for any  $\delta,t$ such that \eqref{ab2} holds.   \qed

\section{Proofs of Lemmas}

\subsection{Proof of Lemma \ref{LA4}} \label{lemmathree}

By Mergelyan's theorem \cite{Mergelyan} it follows that the function $f$ can be approximated by a polynomial such that
\begin{gather} \label{i1}
  \max_{s \in K} \abs{f(s)-p(s)} <\frac{\varepsilon} 3.
 \end{gather}
 Consider the analytic function
\begin{gather*}
 G(s)= \frac{p(s)}{(1+\varepsilon_1 s)^n} 
\end{gather*}
on the half plane $\Re(s)> -{\varepsilon_1}^{-1}$ where $n=\deg p +2$.  It is clear that we may choose $\varepsilon_1>0$ sufficiently small so that
\begin{gather} \label{i2}
  \max_{s \in K} \abs{G(s)-p(s)} <\frac{\varepsilon} 3,
 \end{gather}
 and such that $K$ lies in the half plane $\Re(s)>-\varepsilon_1^{-1}$.
By the theory of Laplace transforms it follows that
\begin{gather} \label{xx1} G(s)=  \int_{0}^\infty e^{-sx} g(x)dx,  \qquad  (s \in K), \\ \intertext{where}
\notag g(x)=  \frac 1 {2 \pi i} \int_{c-\infty i} ^{c+\infty i} e^{sx} G(s)ds, \qquad (c>-\varepsilon_1^{-1}),
\end{gather}
is the inverse Laplace transform of $G$. It is clear that we can truncate the integral \eqref{xx1} so that
\begin{gather} \label{i3} \max_{s \in K} \abs{\int_{A}^{B} g(x) e^{-sx} dx - G(s)}<\frac{\varepsilon} 3.
\end{gather}
The lemma follows by \eqref{i1}, \eqref{i2}, \eqref{i3} and the triangle inequality. \qed

\subsection{Proof  of Lemma \ref{LA0}} \label{mainlemma}

We may find some $0<A<B$ such that 
\begin{gather} \label{we1}
 \max_{s \in K} \abs{\int_A^B \tilde g(x) e^{-sx} dx-f(s)+C}<\frac \varepsilon 7,
\end{gather}
where $\tilde g$ is a $\mathcal C^1$-function which is a suitable smoothed approximation\footnote{If $g$ is $\mathcal C^1$ we may choose $\tilde g=g$.}  of $g$ that satisfies
\begin{gather} \label{tildeg}
\abs{x\tilde g(x)} \leq 1. 
\end{gather}
Since $K$ is compact we may assume that $\delta_1>0$ satisfies
 \begin{gather}  \label{delta1def}
 \min_{s \in K} \Re(1+\delta_1 s) \geq \frac 3 4. \end{gather}
 By using the fact that $\log(1-z)=-z+O(z^2)$ and the fact that $\sum_p p^{-\frac 3 2}$ is convergent it follows that there exists some $P_0$ such that \begin{gather}  \label{we2}
\max_{s \in K} \sup_{|a_p|=1}\sum_{p  \geq P_0}   \abs{  \log \p{1-\frac{a_p}{ p^{1+\delta s}}}+ \frac{a_p}{ p^{1+\delta s}} }  <\frac{\varepsilon} 7, \qquad (0<\delta \leq \delta_1), \\ \intertext{where we may also choose $P_0$ sufficiently large such that} \label{aj}
  \sup_{P_0 \leq N< M} \max_{s \in K} \abs{\sum_{N \leq p_n < M} \frac {(-1)^n}{p_n^{1+\delta s}}}<\frac  \varepsilon 7, \qquad (0<\delta \leq \delta_1),
\end{gather}
where $p_1,p_2,\ldots=2,3,\ldots$ denote the primes in increasing order. 
  Since $\sum_p p^{-1}$ is divergent,  $p^{-1} \to 0$ as $p \to \infty$ and $\log(1-z) \sim -z$ as $z \to 0$  we may choose some $P_1 \geq P_0$ sufficiently large  
  such  that
\begin{gather} \label{we3}
\abs{ 
\sum_{p  < P_1} \log \p{1-\frac{a_p} p}+C}  <\frac{\varepsilon} 7.
\end{gather} 
By the fact that $K$ is compact and continuity it follows that
\begin{gather} \label{we4}
 \max_{s \in K} \abs{ 
\sum_{p  < P_1} \log \p{1-\frac{a_p} p}  - \sum_{p  < P_1} \log \p{1-\frac{a_p} {p^{1+\delta s}}}}  <\frac{\varepsilon} 7,  \qquad (0<\delta\leq \delta_2),
\end{gather}
provided $\delta_2>0$ is sufficiently small. Let us now define
\begin{gather} \label{P1P}
  P_2:=\exp(A \delta^{-1}), \qquad \text{and} \qquad  P_3:=\exp(B \delta^{-1}).
\end{gather}
If $0<\delta \leq \delta_3$ for some sufficiently small $\delta_3>0$ then we may be assured that $P_2 \geq P_1$. 
By the inequality \eqref{tildeg} we may now define 
$$x_n:=\delta \log p_n \tilde g(\delta \log p_n), \qquad  \theta_n:=\arccos(x_n), \qquad \xi_n:=\arg(x_n),$$
 where we choose $\xi_n:=0$ if $x_n=0$, and define for $p_n \geq P_1$
 \begin{gather} \label{apdef}  a_{p_{n}}:=\begin{cases} \exp(i(\xi_{n}+(-1)^n\theta_{n})), &   P_
 2 \leq p_n<P_3, \\ (-1)^n, & P_1 \leq p_n < P_2 \text{ or }  p_n \geq P_3. \end{cases}
 \end{gather} 
Let
 \begin{gather*}  \Lambda_\delta(x):=\sum_{A \leq \delta \log p < x}  \frac{a_p} p-\int_A^x \tilde g(t) dt. 
 \end{gather*}
 By the choice \eqref{apdef}, the estimate $p_{n+1}-p_n \ll \frac{p_n}{(\log p_n)^2}$ for the difference of consecutive primes and the fact that $\tilde g$ is $\mathcal C^1$ we have that
 \begin{gather*}
   \frac 1 2 \p{\frac{a_{p_n}}{p_n}+\frac{a_{p_{n+1}}}{p_{n+1}}}=\frac{\delta \log p_n \tilde g(\delta \log p_n)}{p_n}+O\p{\frac 1 {p_n (\log p_n)}},
 \end{gather*}
  and by invoking the following consequence of the prime number theorem
 \begin{gather*} \sum_{x \leq  \delta \log p \leq y} \frac {\delta \log p} p =y-x+O(\delta),\end{gather*}
 it follows that
  \begin{gather}
     \label{tre}
    \max_{A \leq x \leq B} \abs{\Lambda_\delta(x)} \leq C_1 \delta,
 \end{gather}
 for some $C_1>0$ and all $0<\delta<1$.
 By partial integration we find that
  \begin{gather} \label{thr} \begin{split}
    \sum_{P_2 \leq p< P_3}  \frac{a_p} {p^{1+\delta s}} - \int_A^B \tilde g(x) e^{-sx} dx  &= \int_A^B \Lambda_\delta'(x) e^{-sx} dx, \\ 
   =\Lambda_\delta(B)e^{-Bs}-\Lambda_\delta(A)e^{-As} &+s \int_A^B \Lambda_\delta(x) e^{-sx} dx.
  \end{split}
  \end{gather}
  Thus, since $K$ is compact so that $C_1 \delta (|s|(B-A)+1) \p{\abs{e^{-sA}}+\abs{e^{-sB}}}$ is bounded for $s \in K$ and can be as small as we wish provided $\delta$ is small enough,  it follows from \eqref{tre} and \eqref{thr} that
   \begin{gather} \label{we5} \max_{s \in K} \abs{\sum_{P_2 \leq p< P_3}  \frac{a_p} {p^{1+\delta s}} - \int_A^B \tilde g(x) e^{-sx} dx} < \frac{\varepsilon} 7, \qquad(0<\delta \leq \delta_4), \end{gather}
   for some sufficiently small $\delta_4>0$. Since $P_0 \leq P_1<P_3$ It is clear by \eqref{delta1def},\eqref{aj} and the definition of  $a_p$, Eq \eqref{apdef} for $P_1 \leq p< P_2$ and $p \geq P_3$ 
    that
\begin{gather} \label{we6}
  \max_{s \in K} \abs{\sum_{P_1 \leq p <P_2} \frac{a_p}{p^{1+\delta s}}} <\frac{\varepsilon} 7, \qquad  \max_{s \in K} \abs{\sum_{p\geq P_3} \frac{a_p}{p^{1+\delta s}}} <\frac{\varepsilon} 7, 
  \end{gather}
for $0<\delta \leq \delta_1$.
 By the choice \eqref{apdef} of $a_p$ for $p \geq P_3$ it is also clear that 
\begin{gather*}
  h(s):=-\sum_{p} \log \p{1-\frac{a_p}{ p^{s}}}
\end{gather*}
converges to an analytic function for $\Re(s)>\frac 1 2 $.
Finally our lemma follows with $\delta_0:=\min(1,\delta_1,\delta_2,\delta_3,\delta_4)$  by the inequalities \eqref{we1}, \eqref{we2},  \eqref{we3}, \eqref{we4}, \eqref{we5}, \eqref{we6}  and the triangle inequality. \qed

We would finally like to remark that instead of using the definition \eqref{apdef} for $a_p$ in the intermediate range $P_2 \leq p <P_3$ the alternative recursive definition
 \begin{gather*}
 a_p:=\exp\left(i \arg  \left( \int_A^{\delta \log p}  g(x) dx -\sum_{P_2 \leq q<p}  \frac{a_q} {q}  \right) \right)
\end{gather*}
gives the same result\footnote{In such  case it is not necessary to have a $\mathcal C^1$-function so we may replace $\tilde g$ with $g$ in \eqref{we1}.}.  We will give yet another way to prove this in the proof of  \cite[Lemma 2]{Andersson2} where we prove some corresponding joint universality results for Dirichlet 
$L$-functions.
\section{Final discussion}

Finally we would like to give four remarks on our main result, Theorem \ref{th}, and its proof.
\begin{enumerate}
\item  If we choose the compact set $K$ in the half-plane $\{s \in \C:\Re(s)>0\}$, then Theorem \ref{th}  gives us a universality theorem in the half-plane of absolute convergence, which to our knowledge is the second theorem of such type\footnote{The present author also has had (for some years...) a third result of this type in progress. To appear eventually...}, where the first is a theorem on the difference of two Epstein zeta-functions, and the universality result is in the lattice aspect \cite[Theorem 1.10]{AndSod}. Also it is clear that the same proof method  allows us to obtain the corresponding result  for any Dirichlet series with Euler product
\begin{gather*}
  L(s)= \prod_{p} \sum_{k=0}^\infty c_{p^k} p^{-ks} 
\\ \intertext{if there exist some $c>0$ such that  }
  \liminf_{N \to \infty} \sum_{N<p<N^{1+\xi}} \frac {|c_p| } p \geq \log(1+\xi) c, 
  \\ \intertext{for any $\xi>0$,}
  \lim_{p \to \infty} \frac{|c_p|} p=0, \qquad \sum_p \sum_{k=2}^\infty  \frac {|c_{p^k}| k \log p } {p^k}<\infty,
\end{gather*}
 and  $L(s)$ has abscissa of convergence $1$.  We do not need that the Dirichlet series $L(s)$ has an analytic continuation to $\Re(s)\leq 1$. It is sufficient that it is analytic for $\Re(s)>1$, as long as we require the compact set to lie in the half plane $\{s \in \C: \Re(s)>0\}$. We will further develop this idea in \cite{Andersson29}, where we also obtain joint universality results.
\item An advantage to our method is that it should not be too difficult to make it effective. The Pechersky rearrangement theorem is notoriously ineffective and the only effective method so far for proving universality is the one of Good \cite{Good}, further developed by Garunk\v{s}tis  \cite{Garunkstis}. 
\item The same general idea can be used to obtain universality results for the Hurwitz zeta-function. A remarkable feature is that this approach will also be able to handle algebraic irrational parameters. This will be further developed in  \cite{Andersson21}.
\item We may also use the main idea of this paper to prove universality in the family aspect of $L$-functions. We will further develop this idea in \cite{Andersson25}
\end{enumerate}

\bibliographystyle{plain}

\end{document}